\documentclass[11pt]{amsart}
\usepackage[utf8]{inputenc}
\usepackage{amsmath, amssymb, amsthm}
\usepackage[a4paper,margin=3cm]{geometry}
\usepackage{comment}
\usepackage{aliascnt}

\usepackage{paralist}

\theoremstyle{definition}
\newtheorem{theorem}{Theorem}[section]

\newcommand{\newaliasthm}[2]{%
  \newaliascnt{#1}{theorem}%
  \newtheorem{#1}[#1]{#2}%
  \aliascntresetthe{#1}%
}

\newaliasthm{lemma}{Lemma}
\newaliasthm{claim}{Claim}
\newaliasthm{corollary}{Corollary}
\newaliasthm{conjecture}{Conjecture}
\newaliasthm{proposition}{Proposition}
\newaliasthm{question}{Question}
\newaliasthm{problem}{Problem}
\newaliasthm{example}{Example}
\newaliasthm{remark}{Remark}
\newaliasthm{definition}{Definition}
\newaliasthm{observation}{Observation}

\usepackage{xcolor}
\definecolor{mcitecolor}{HTML}{5590B4}

\usepackage{hyperref}
\hypersetup{
linkcolor  = black,
citecolor  = mcitecolor!85!black,
urlcolor   = blue!85!black,
colorlinks = true,
}

\usepackage[
  capitalise,
  noabbrev,
  nameinlink,
  nosort
]{cleveref}

\newcommand{\Dih}{\operatorname{Dih}}
\newcommand{\Cay}{\operatorname{Cay}}
\newcommand{\Sch}{\operatorname{Sch}}
\newcommand{\Z}{\mathbb{Z}}
\newcommand{\R}{\mathbb{R}}

\title[Coloring semiminimal Cayley Graphs]{Coloring semiminimal Cayley Graphs}

\author{Ignacio Garc\'ia-Marco}
\address[IGM]{Facultad de Ciencias and Instituto de Matem\'aticas y Aplicaciones (IMAULL), Universidad de La Laguna, Apartado de Correos 456, 38200 La Laguna, Spain}
\email{iggarcia@ull.edu.es}

\author{Kolja Knauer}
\address[KK]{Universitat de Barcelona,
Departament de Matemàtiques i Informàtica, 08007 Barcelona, Spain}
\email{kolja.knauer@ub.edu}

\author{Giuliamaria Menara}
\address[GM]{Institut de Mathématiques de Jussieu - Paris Rive Gauche, 75252 Paris, France}
\email{giuliamaria.menara@imj-prg.fr}

\begin{document}

\begin{abstract}

In 1978 Babai raised the question whether all semiminimal Cayley graphs have bounded chromatic number. In this paper we show that semiminimal Cayley graphs of abelian and generalized dihedral groups have {circular} {chromatic} number at most 4, thus extending a result of Barajas and Serra. 
\end{abstract}

\maketitle


\section{Introduction}
Given a group $G$ and a \emph{connection set} $C\subseteq G$, the (undirected, right) \emph{Cayley graph} $\Cay(G,C)$ has vertex set $G$, with two distinct elements $a,b\in G$ adjacent whenever
$
a^{-1}b\in C\cup C^{-1}
$.
Thus, throughout the paper, Cayley graphs are regarded as undirected, even when the connection set is not inverse-closed.

A \emph{semiminimal\footnote{also called \emph{quasi-irreducible} in the literature}  family} of $G$ is a family $
C=(c_\alpha)_{\alpha<\iota}
$ of elements of $G$
indexed by an ordinal $\iota$, such that
\[
c_\alpha\notin
\langle c_\beta:\beta<\alpha\rangle
\qquad\text{for every }\alpha<\iota.
\]

Babai~\cite{B78} shows that every group admits a semiminimal generating family. A Cayley graph $\Cay(G,C)$ is \emph{semiminimal} if $C$ is a semiminimal generating family of $G$. An important special case is that of \emph{minimal}, or \emph{irreducible}, Cayley graphs, for which $C$ is an inclusion-minimal generating family of  $G$. Some structural properties of minimal Cayley graphs have been studied~\cite{GMK22,God81,KS25,Spe83}, but much remains to be understood in both the minimal and the semiminimal setting.

A question raised by Babai~\cite{B78,B95} asks whether minimal Cayley graphs, or more generally semiminimal Cayley graphs, have uniformly bounded chromatic number. These questions have recently received renewed attention because of the structural and computational challenges they present. It was proved in~\cite{GMK25} that minimal Cayley graphs of finitely generated nilpotent and of finitely generated generalized dihedral groups have chromatic number at most $3$, whereas $4$ colors are sometimes necessary for soluble groups. This result motivated the construction of the database~\cite{Rhys}, which was used to verify computationally that every minimal Cayley graph of a group of order at most $512$ has chromatic number at most $4$. By contrast, a broad, purely graph-theoretic generalization of Babai's question was answered negatively in~\cite{DHY24}.


For semiminimal Cayley graphs, Barajas and Serra~\cite{barajas2009chromatic} proved that those of cyclic groups have circular chromatic number at most $4$, which is best possible. Here, the circular chromatic number is a continuous parameter whose ceiling upper bounds the chromatic number, see~\cite{Zhu01}. On the other hand, generalized quaternion groups admit semiminimal Cayley graphs of chromatic number $7$, see~\cite{GMK25}. Babai~\cite{B78} showed that every group admits a semiminimal Cayley graph of chromatic number at most countable.

In the present paper, we extend the known positive results. We first prove that every semiminimal Cayley graph of a finitely generated abelian group has circular chromatic number at most $4$, see \Cref{thm:abelian-circular}. 
We then extend the proof method  
to the case of finitely generated generalized dihedral groups, where we also prove the upper bound of $4$ in \Cref{thm:semiminimal-Dih-circular}.

Using compactness arguments for the circular chromatic number, we extend the bound to arbitrary abelian groups and arbitrary generalized dihedral groups, see \Cref{cor:arbitrary-abelian-circular} and \Cref{cor:arbitrary-dihedral-circular}.

Finally, we obtain a uniform bound for the chromatic number of semiminimal Cayley graphs of bounded Jordan constant groups, see \Cref{cor:easy 4k bound for non-Ab}, which can be further improved in the case of semidirect products of abelian groups, see \Cref{thm:abelian-by-abelian-action-bound}.

\subsection{Preliminaries}

Recall that a \emph{(well-ordered) semiminimal family} of a group $G$ is a family $
C=(c_\alpha)_{\alpha<\iota}
$ of elements of $G$
indexed by an ordinal $\iota$, such that
\[
c_\alpha\notin
\langle c_\beta:\beta<\alpha\rangle
\qquad\text{for every }\alpha<\iota.
\]
\begin{observation}\label{obs:semimin}
    Let $
C=(c_\alpha)_{\alpha<\iota}
$ be a semiminimal generating family of $G$, then:
\begin{enumerate}[(i)]
    \item Every finite
subfamily $D$ of $C$, with the inherited order, is a semiminimal 
family of $G$.
    \item Every finite
semiminimal family $D$ of $G$ can be extended to a semiminimal generating 
family of $G$.
\end{enumerate}
\end{observation}
\begin{proof}
Since (i) is trivial, we only prove (ii). We extend $D$ to a well-ordered semiminimal generating family of $G$. To this
end, well-order $G\setminus D$ and consider its elements successively.
Proceeding by transfinite recursion, retain an element whenever it does not belong to the subgroup generated by $D$ and the elements retained previously. Denote the resulting family,
preceded by $D$, by $C$.
By construction, $C$ is semiminimal. It also generates $G$, since every
element that was not retained already belonged to the subgroup generated by
the elements considered previously. 
\end{proof}

A \emph{proper coloring} of a graph $X = (V,E)$ is an assignment of colors to its vertices such that adjacent vertices receive distinct colors. The \emph{chromatic number} $\chi(X)$ is the minimum number of colors required for a proper coloring of $X$.
The \emph{circular chromatic number} $\chi_c(X)$ of a graph $X=(V,E)$ is a continuous analogue of the chromatic number introduced in~\cite{Vin88} under the name of \emph{star chromatic number}. It is defined as the infimum of all real numbers $s\geq 1$ for which there exists a map
\[
f:V\longrightarrow \R/\Z
\]
such that, for every edge $\{u,v\}$, the circular distance between $f(u)$ and $f(v)$ is at least $1/s$.
The chromatic number and the circular chromatic number both are \emph{monotone}, i.e., they cannot increase under taking subgraphs. Further, they satisfy
\[
\chi(X)=\left\lceil \chi_c(X)\right\rceil.
\]

Hence, all our results yield upper bounds on the chromatic number. See~\cite{Zhu01} for more information on the circular chromatic number. 

\section{Finitely Generated Abelian Groups}

In~\cite[Theorem 8]{barajas2009chromatic} Barajas and Serra proved that for $A=\Z_n$, $\chi_c\bigl(\Cay(\Z_n, C)\bigr) \le 4$ for every semiminimal family. In this section we extend this result for $A$ a finitely generated abelian group.

\begin{definition}
\label{def:Hi(D) and r(D)}
Let $A$ be an abelian group and $
C=(c_\alpha)_{\alpha<\iota}
$
 a well-ordered family of elements of $A$. For every
$\alpha<\iota$, set
$
H_\alpha:=\langle c_\beta:\beta<\alpha\rangle;
$
the \emph{relative order} of $c_\alpha$ in $C$ is
\[
r_\alpha:=\operatorname{ord}_{A/H_\alpha}(c_\alpha+H_\alpha) \in \Z_{> 0} \cup \{\infty\}.\] 

We write
\[
r(C):=\min_{\alpha < \iota}r_\alpha,
\qquad
r(\varnothing):=\infty,
\]
and call it the \emph{order} of $C$.
\end{definition}

We observe that $C$ is a semiminimal family if and only if $r(C) \geq 2$.

The following lemma isolates the character-extension argument that will be
used both for abelian and generalized dihedral groups.

\begin{lemma}
\label{lem:character-prescribed-arcs}
Let $D=(d_1,\ldots,d_m)$ be a generating family of an
abelian group $A$, and for each $i\in\{1,\ldots,m\}$ let $r_i$ denote
the relative order of $d_i$ in $D$. For each $i$, consider a subset
$I_i\subseteq\mathbb R/\mathbb Z$ defined as follows:
\begin{itemize}
    \item if $r_i\in\mathbb Z_{>0}$, let $I_i$ be a closed
    arc with
    $$
    \operatorname{length}(I_i)\geq\frac1{r_i};
    $$
    \item if $r_i=\infty$, let $I_i$ be a nonempty set.
\end{itemize}
Then there exists a character
$$
\varphi:A\longrightarrow\mathbb R/\mathbb Z
$$
such that $\varphi(d_i)\in I_i$ for every $i\in\{1,\ldots,m\}$.

\end{lemma}

\begin{proof}
Set
$
H_i(D):=\langle d_1,\ldots,d_{i}\rangle
$
for every $i\in\{1,\ldots,m+1\}$. 
We construct compatible characters
$$
\varphi_i:H_i(D)\longrightarrow\mathbb R/\mathbb Z
$$
inductively. For $H_0(D)=\{0\}$, take the trivial character. Suppose that
$\varphi_{i-1}$ has been constructed. 

If $r_i = \operatorname{ord}_{A/H_{i-1}(D)}(d_i+H_{i-1}(D))  <\infty$, then an extension to
$H_i(D)$ is determined by a value $y:=\varphi_i(d_i)$ satisfying
$$
r_i y=\varphi_{i-1}(r_i d_i).
$$
This equation has exactly $r_i$ solutions in $\mathbb R/\mathbb Z$, equally
spaced by $1/r_i$. Hence, the closed arc $I_i$ contains one of
these solutions. 
Choose such a solution $y$ in $I_i$ and define
$$
\varphi_i(h+kd_i):=\varphi_{i-1}(h)+ky
$$
for $h\in H_{i-1}(D)$ and $k\in\mathbb Z$. 
This is well-defined. Indeed, if
$$
h+kd_i=h'+k'd_i,
$$
then $(k-k')d_i\in H_{i-1}(D)$. Since $d_i+H_{i-1}(D)$ has order $r_i$ in
$A/H_{i-1}(D)$, we have $r_i\mid(k-k')$. Writing $k-k'=qr_i$ and using
$$
h'-h=(k-k')d_i=qr_i d_i,
$$
we obtain
\begin{align*}
\varphi_{i-1}(h')-
\varphi_{i-1}(h)
&=q\varphi_{i-1}(r_i d_i)\\
&=qr_i y\\
&=(k-k')y,
\end{align*}
which gives
$$
\varphi_{i-1}(h)+ky=\varphi_{i-1}(h')+k'y.
$$

If $r_i=\infty$, then no nonzero multiple of $d_i$ belongs to $H_{i-1}(D)$.
Hence
$$
H_i(D)=H_{i-1}(D)\oplus\langle d_i\rangle,
$$
and we may choose $\varphi_i(d_i)$ as any element of $I_i$.

At the end we obtain a character $\varphi=\varphi_m:A\to\mathbb R/\mathbb Z$
with the desired properties.
\end{proof}

\begin{theorem}
\label{thm:abelian-circular}
Let $A$ be an abelian group and let $C=(c_1,\ldots,c_m)$ be
a semiminimal family. If $r := r(C)$ is the order of $C$, then
$$
\chi_c(\Cay(A,C))\leq
\begin{cases}
2,&\text{if }r = \infty,\\[1mm]
\dfrac{2r}{r-1},&\text{if }r \in\mathbb N.
\end{cases}
$$
\end{theorem}

\begin{proof}
We first observe that all the connected components of $\Cay(A,C)$ are isomorphic to $\Cay(B,C)$, where $B = \langle c_1,\ldots,c_m\rangle$ is the group spanned by the elements of $C$. Hence $\chi_c(\Cay(A,C)) = \chi_c(\Cay(B,C))$.

Assume first that $r = \infty$, then the relative order of $c_i$ in $C$ is $r_i = \infty$. We set $I = \{\frac{1}{2}\}$. By
\Cref{lem:character-prescribed-arcs}, there exists a character
$\varphi:B\to\mathbb R/\mathbb Z$ such that $\varphi(c_i) \in I$. Consequently, if $x$ and $y$ are adjacent in $\Cay(B,C)$, then $y-x=\pm c_i$ for some $i$,
and hence
$$
\|\varphi(y)-\varphi(x)\|=\|\varphi(c_i)\| = \frac12.
$$
Therefore
$$
\chi_c(\Cay(B,C))\leq 2.
$$

Assume now that $r \in \mathbb N$, then the relative order of $c_i$ in $C$ is $r_i \geq r$. Also, since $C$ is semiminimal, we have that $r \geq 2$. We set $I$ the closed arc 
$$
I:=\left[\frac12-\frac1{2r},\frac12+\frac1{2r}\right]
\subseteq\mathbb R/\mathbb Z.
$$
of length $1/r \geq 1/r_i$. By
\Cref{lem:character-prescribed-arcs}, there exists a character
$\varphi:B\to\mathbb R/\mathbb Z$ such that $\varphi(c_i) \in I$. Consequently, 
if $x$ and $y$ are adjacent in $\Cay(B,C)$, then $y-x=\pm c_i$ for some $i$,
and hence
$$
\|\varphi(y)-\varphi(x)\|=\|\varphi(c_i)\|\geq\frac12-\frac1{2r} = \frac{r-1}{2r}.
$$
Therefore
$$
\chi_c(\Cay(B,C))\leq\frac{2r}{r-1}.
$$
In both cases, combining the bound obtained for $\chi_c(\Cay(B,C))$
with the equality $\chi_c(\Cay(A,C))=\chi_c(\Cay(B,C))$, established
above, gives the claimed bound for $\chi_c(\Cay(A,C))$.
\end{proof}

\section{Finitely Generated Generalized Dihedral Groups}

Let $A$ be an abelian group. The generalized dihedral group associated to
$A$ is the semidirect product
$$
\Dih(A)=A\rtimes C_2,
$$
where the nontrivial element of $C_2$ acts on $A$ by inversion. We write
elements of $\Dih(A)$ as pairs $(a,\varepsilon)$, with $a\in A$ and
$\varepsilon\in\{0,1\}$, and multiplication
$$
(a,\varepsilon)(b,\delta)
=\bigl(a+(-1)^\varepsilon b,\varepsilon+\delta\bigr),
$$
where the second coordinate is taken modulo $2$. The elements of
$A\times\{0\}$ are called \emph{rotations}, while the elements of $A\times\{1\}$
are called \emph{reflections}.

\begin{lemma}
\label{lem:normalize-reflection}
Let $A$ be an abelian group and let $\tau=(a,1)\in\Dih(A)$ be a reflection.
Then the map
$$
\alpha_a:\Dih(A)\longrightarrow\Dih(A),\qquad
\alpha_a(x,\varepsilon):=(x+\varepsilon a,\varepsilon),
$$
is a group automorphism with $\alpha_a(0,1)=\tau$.
\end{lemma}

\begin{proof}
Let $\overline{\varepsilon+\delta}$ denote the sum of
$\varepsilon,\delta\in\{0,1\}$ modulo $2$. A direct computation gives
\begin{align*}
\alpha_a\bigl((x,\varepsilon)(y,\delta)\bigr)
&=\bigl(x+(-1)^\varepsilon y+
\overline{\varepsilon+\delta}\,a,
\overline{\varepsilon+\delta}\bigr),\\
\alpha_a(x,\varepsilon)\alpha_a(y,\delta)
&=\bigl(x+(-1)^\varepsilon y+
(\varepsilon+(-1)^\varepsilon\delta)a,
\overline{\varepsilon+\delta}\bigr).
\end{align*}
These expressions agree because
$$
\overline{\varepsilon+\delta}
=\varepsilon+(-1)^\varepsilon\delta
$$
for $\varepsilon,\delta\in\{0,1\}$. Thus $\alpha_a$ is a homomorphism. Its
inverse is $\alpha_{-a}$. Finally, $\alpha_a(0,1)=(a,1)=\tau$.
\end{proof}

\begin{lemma}
\label{lem:projected-semiminimal}
Let $A$ be an abelian group and $C=(c_1,\ldots,c_m)$ a semiminimal family of~$\Dih(A)$,
where $c_i=(d_i,\varepsilon_i)$ for every $i\in\{1,\ldots,m\}$. Suppose $c_s=(0,1)$ is the first reflection
occurring in $C$.
Set
$$
D:=(d_i)_{i\neq s},
$$
with the order inherited from $C$, and write $B:=\langle d_i : i\neq s\rangle\leq A$
for the subgroup it generates. Then $D$ is a semiminimal family of
$A$, and
$$
\Dih(B)=\langle c_1,\ldots,c_m\rangle.
$$

\end{lemma}

\begin{proof}
For every $i\neq s$, let
$$
B_i:=\langle d_j:j<i,\ j\neq s\rangle.
$$
We first show that $d_i\notin B_i$ for every $i\neq s$.

Suppose first that $i<s$. By the choice of $s$, all the elements $c_j$ with
$j\leq i$ are rotations. Hence $c_j=(d_j,0)$ for every $j\leq i$. If
$d_i\in B_i$, then
$$
c_i=(d_i,0)\in\langle c_j:j<i\rangle,
$$
contrary to the semiminimality of $C$.

Now suppose that $i>s$. Then $\tau:=c_s=(0,1)$ belongs to the subgroup
generated by the predecessors of $c_i$. Moreover, for every $j<i$ with
$j\neq s$, the rotation $(d_j,0)$ also belongs to this subgroup. Indeed, if
$c_j$ is a rotation, then $(d_j,0)=c_j$, while if $c_j$ is a reflection,
then $(d_j,0)=c_j\tau$. Consequently, if $d_i\in B_i$, then
$$
(d_i,0)\in\langle c_j:j<i\rangle.
$$
If $c_i$ is a rotation, this directly implies that $c_i$ belongs to the
subgroup generated by its predecessors. If $c_i$ is a reflection, then
$$
c_i=(d_i,1)=(d_i,0)\tau
$$
also belongs to that subgroup. Both conclusions contradict the
semiminimality of $C$. Thus $D$ is semiminimal.

Finally, we verify $\Dih(B)=\langle c_1,\ldots,c_m\rangle=:\Gamma$.
The computation above shows that $(d_i,0)\in\Gamma$ for every
$i\neq s$ (namely $(d_i,0)=c_i$ if $c_i$ is a rotation, and
$(d_i,0)=c_i\tau$ if $c_i$ is a reflection). Since $\Gamma$ is a
subgroup, the elements $(d_i,0)$, $i\neq s$, together generate
$B\times\{0\}$ inside $\Gamma$, and $\tau=(0,1)\in\Gamma$; hence
$\Dih(B)=(B\times\{0\})\cup(B\times\{0\})\tau\subseteq\Gamma$.
Conversely, $\tau\in\Dih(B)$ trivially, and $c_i=(d_i,\varepsilon_i)
\in\Dih(B)$ for every $i\neq s$ since $d_i\in B$ by definition of
$B$; as these elements generate $\Gamma$, we get
$\Gamma\subseteq\Dih(B)$. Combining both inclusions gives
$\Dih(B)=\Gamma$.
\end{proof}

We can now prove the main result of this section.

\begin{theorem}
\label{thm:semiminimal-Dih-circular}
Let $A$ be an abelian group and $C=(c_1,\ldots,c_m)$ 
a semiminimal family of~$\Dih(A)$. Write
$c_i=(a_i,\varepsilon_i)$, and let $s$ be the first index such that
$\varepsilon_s=1$. 
Define
$$
D:=(d_i)_{i\neq s}, \text{ where }
d_i:=a_i-\varepsilon_i a_s.
$$
 If $r:=r(D)$ is the order of $D$, then
$$
\chi_c(\Cay(\Dih(A),C))\leq
\begin{cases}
2,&\text{if }r=\infty,\\[1mm]
\dfrac{2r}{r-1},&\text{if }r\in\mathbb N.
\end{cases}
$$
\end{theorem}

\begin{proof}
Recall that $c_s=(a_s,1)$ denotes the first reflection in $C$. By
\Cref{lem:normalize-reflection}, the automorphism $\alpha_{-a_s}$
induces an isomorphism
$$
\Cay(\Dih(A),C)
\cong
\Cay\bigl(\Dih(A),\alpha_{-a_s}(C)\bigr).
$$
Replacing $C$ by its image and retaining the same notation, we may
assume that $c_s=(0,1)$; under this normalization, the first
coordinate of $c_i$ becomes $a_i-\varepsilon_i a_s=d_i$, matching the
family $D$ defined above.

Set $\Gamma:=\langle c_1,\ldots,c_m\rangle$. Every connected component
of $\Cay(\Dih(A),C)$ is isomorphic to $\Cay(\Gamma,C)$, and hence
$$
\chi_c(\Cay(\Dih(A),C))=\chi_c(\Cay(\Gamma,C)).
$$
By \Cref{lem:projected-semiminimal}, $D=(d_i)_{i\neq s}$ is a
semiminimal family of $A$. Moreover, setting
$B:=\langle d_i : i\neq s\rangle$, the same lemma gives
$\Dih(B)=\Gamma$.

Assume first that $r\in\mathbb N$. For each $i\neq s$, define an arc
$I_i\subseteq\mathbb R/\mathbb Z$ by
$$
I_i:=
\begin{cases}
I_R :=\left[\frac12-\frac1{2r},\frac12+\frac1{2r}\right],&\text{if }c_i\text{ is a rotation},\\
I_F:=\left[-\frac1{2r},\frac1{2r}\right],&\text{if }c_i\text{ is a reflection}.
\end{cases}
$$
Both $I_R$ and $I_F$ have length $1/r$. Writing $r_i$ for the
relative order of $d_i$ in $D$, we have $r_i\geq r$ for every
$i\neq s$, so \Cref{lem:character-prescribed-arcs} gives a character
$$
\varphi:B\longrightarrow\mathbb R/\mathbb Z
$$
with $\varphi(d_i)\in I_i$ for every $i\neq s$. Define
$$
\Psi:\Dih(B)\longrightarrow\mathbb R/\mathbb Z,
\qquad
\Psi(b,\varepsilon):=\varphi(b)+\frac{\varepsilon}{2}.
$$

We show that adjacent vertices are assigned values at circular
distance at least $(r-1)/(2r)$. Every edge can, after possibly
exchanging its endpoints, be written as $\{x,xc_i\}$ for some
$x\in\Dih(B)$ and $c_i\in C$; write $x=(b,\varepsilon)$.

If $c_i=(d_i,0)$ is a rotation generator, then
$$
\Psi(xc_i)-\Psi(x)=(-1)^\varepsilon\varphi(d_i).
$$
Since $\varphi(d_i)\in I_R$,
$$
\|\Psi(xc_i)-\Psi(x)\|=\|\varphi(d_i)\|\geq\frac{r-1}{2r}.
$$

If $c_i=c_s=(0,1)$, then
$$
\|\Psi(xc_i)-\Psi(x)\|=\frac12\geq\frac{r-1}{2r}.
$$

Finally, let $c_i=(d_i,1)\neq c_s$ be a reflection generator. Then, in
$\mathbb R/\mathbb Z$,
$$
\Psi(xc_i)-\Psi(x)=(-1)^\varepsilon\varphi(d_i)+\frac12.
$$
Since $\varphi(d_i)\in I_F$, we obtain
$(-1)^\varepsilon\varphi(d_i)+\frac12\in I_R$, so this edge also moves
by circular distance at least $(r-1)/(2r)$.

Consequently,
$$
\chi_c(\Cay(\Dih(B),C))\leq\frac{2r}{r-1}.
$$

Suppose now that $r=\infty$. For each $i\neq s$, define
$I_i:=\{1/2\}$ if $c_i$ is a rotation, and $I_i:=\{0\}$ if $c_i$ is a
reflection. Since every $r_i$ is infinite, \Cref{lem:character-prescribed-arcs} yields a character
$\varphi:B\to\mathbb R/\mathbb Z$ with $\varphi(d_i)\in I_i$ for every
$i\neq s$; define $\Psi$ as above. Every rotation generator changes
the value of $\Psi$ by $\pm1/2$, and every reflection generator
changes it by $1/2$. Hence every edge has endpoints whose images are
at circular distance $1/2$, so
$$
\chi_c(\Cay(\Dih(B),C))\leq2.
$$

In both cases, combining the bound obtained for
$\chi_c(\Cay(\Dih(B),C))$ with the equality
$\chi_c(\Cay(\Dih(A),C))=\chi_c(\Cay(\Gamma,C))$, established above,
gives the claimed bound for $\chi_c(\Cay(\Dih(A),C))$.
\end{proof}

\section{Locality of finite chromatic bounds}

The following compactness result is the circular analogue of the
de Bruijn--Erd\H{o}s theorem~\cite{deBruijnErdos1951}. 

\begin{lemma}
\label{lem:circular-compactness}
Let $X$ be a graph and let $M\geq 1$ be a real number. Then $X$ admits a
circular $M$-coloring if and only if every finite subgraph of $X$ admits a
circular $M$-coloring.
Consequently,
$$
\chi_c(X)
=
\sup\left\{
\chi_c(Y):
Y\text{ is a finite subgraph of }X
\right\}.
$$
\end{lemma}

\begin{proof}
The forward implication is immediate. Conversely, let
$$
\mathbb T:=\mathbb R/\mathbb Z.
$$
Since $\mathbb T$ is compact, Tychonoff's theorem~\cite{MunkresTopology}
implies that the product space
$$
\mathbb T^{V(X)}
=
\prod_{v\in V(X)}\mathbb T
$$
is compact. For every edge $uv\in E(X)$, define
$$
F_{uv}
:=
\left\{
f\in\mathbb T^{V(X)}:
\|f(u)-f(v)\|\geq\frac1M
\right\},
$$
where $\|x\|$ denotes the circular distance from $x\in\mathbb T$ to $0$. Since the map
$
f\longmapsto\|f(u)-f(v)\|
$
is continuous, 
the set $F_{uv}$ is closed.

We claim that the family $(F_{uv})_{uv\in E(X)}$
has the \emph{finite-intersection property}, i.e., all finite intersections are non-empty. Indeed, let
$u_1v_1,\ldots,u_mv_m$ be finitely many edges of $X$. The graph formed by these edges and their
endpoints is finite and therefore admits a circular $M$-coloring by
assumption. Extending this coloring arbitrarily to the remaining vertices of
$X$ gives an element of
$
\bigcap_{i=1}^m F_{u_iv_i}.
$

Since $\mathbb T^{V(X)}$ is compact, the finite-intersection property for
closed sets gives
$$
\bigcap_{uv\in E(X)}F_{uv}\neq\varnothing.
$$
Every element of this intersection is a circular $M$-coloring of $X$.

It remains to prove the equality concerning the circular chromatic number.
Set
$$
s := \sup\left\{ \chi_c(Y): Y\text{ is a finite subgraph of }X \right\}.
$$
By the first part of the lemma, we have that $\chi_c(X)\leq s$. Since circular chromatic number is monotone under taking subgraphs, we have
$
s\leq\chi_c(X).
$
\end{proof}

\begin{lemma}
\label{thm:circular-locality-fixed-family}
Let $G$ be a group and let $C$ be a well-ordered semiminimal  family of $G$. Then
$$
\chi_c(\Cay(G,C))
=
\sup\left\{
\chi_c\bigl(\Cay(\langle D\rangle,D)\bigr):
D\text{ is a finite subfamily of }C
\right\}.
$$
\end{lemma}

\begin{proof}
Let
$
s
:=
\sup\left\{
\chi_c\bigl(\Cay(\langle D\rangle,D)\bigr):
D\text{ is a finite subfamily of }C
\right\}.
$
For every finite subfamily $D$ of $C$, clearly $
s\leq\chi_c(\Cay(G,C))
$ holds.

For the reverse inequality, let $X$ be a finite subgraph of $\Cay(G,C)$.
Only finitely many elements of $C$ occur as labels of edges of $X$. Let
$
D=(d_1,\ldots,d_m)
$
be the corresponding finite subfamily of $C$, with the order inherited from
$C$, and set
$
L:=\langle D\rangle
$.

Every connected component of $X$ is contained in a single left coset
of $L$: if a component contains a vertex $x$, all of its vertices lie
in $xL$.

Left multiplication by $x^{-1}$ maps the part of $X$ contained in $xL$
isomorphically to a subgraph of
$
\Cay(L,D)
$.
It follows that each connected component of $X$ has circular chromatic
number at most
$
\chi_c(\Cay(L,D))\leq s
$. Hence, $
\chi_c(X)\leq s$ and \Cref{lem:circular-compactness} gives
$$
\chi_c(\Cay(G,C))
=
\sup\left\{
\chi_c(X):
X\text{ is a finite subgraph of }\Cay(G,C)
\right\}
\leq s.
$$
\end{proof}

\begin{theorem}
\label{cor:arbitrary-abelian-circular}
Let $A$ be an abelian group and $
C=(c_\alpha)_{\alpha<\iota}
$
 a well-ordered semiminimal  family. If $r := r(C)$ denotes the order of $C$, then
$$
\chi_c\bigl(\Cay(A,C)\bigr)
\leq
\begin{cases}
2,&r=\infty,\\[2mm]
\dfrac{2r}{r-1},&r<\infty.
\end{cases}
$$
\end{theorem}

\begin{proof}
 For every
$\alpha<\iota$, set
$
H_\alpha:=\langle c_\beta:\beta<\alpha\rangle,
$
and 
$
r_\alpha:=\operatorname{ord}_{A/H_\alpha}(c_\alpha+H_\alpha)$ the relative order of $c_\alpha$ in $C$. 
Let $D$ be a finite subfamily of $C$, with the order inherited from $C$.
For each $c_\alpha\in D$, the subgroup generated by its predecessors in $D$
is contained in $H_\alpha$. Consequently, the relative order of $c_\alpha$ modulo
the subgroup generated by its predecessors in $D$ is at least $r_\alpha$,
and hence at least $r$. Since $r\mapsto\frac{2r}{r-1}$ is non-increasing
for $r\geq2$, the inequality $r(D)\geq r$ yields the bound stated in
terms of $r$. The result follows from \Cref{thm:abelian-circular} and 
\cref{thm:circular-locality-fixed-family}.
\end{proof}

\begin{theorem}
\label{cor:arbitrary-dihedral-circular}
Let $A$ be an abelian group, and let $C=(c_\alpha)_{\alpha<\iota}$ be
a well-ordered semiminimal  family of $\Dih(A)$. Write
$c_\alpha=(a_\alpha,\varepsilon_\alpha)$ for each $\alpha<\iota$, and
let $s<\iota$ be the least index with $\varepsilon_s=1$. Define
$$
D:=(d_\alpha)_{\substack{\alpha<\iota\\ \alpha\neq s}},
\text{ where }
d_\alpha:=a_\alpha-\varepsilon_\alpha a_s.
$$
If $r:=r(D)$ denotes the order of $D$, then
$$
\chi_c(\Cay(\Dih(A),C))\leq
\begin{cases}
2, & \text{if } r=\infty,\\[1mm]
\dfrac{2r}{r-1}, & \text{if } r\in\mathbb N.
\end{cases}
$$
\end{theorem}

\begin{proof}
Let $X$ be a finite subgraph of $\Cay(\Dih(A),C)$. Only finitely many elements of $C$ occur as edge labels in
$X$. Let $E\subseteq C$ be the corresponding finite subfamily, with
the order inherited from $C$, and set $E':=E\cup\{c_s\}$, 
again with the order inherited from $C$. Then by Observation \ref{obs:semimin} $E'$ is a finite
semiminimal generating family, and every connected component of $X$
is, up to left translation, a subgraph of $\Cay(\langle
E'\rangle,E')$. Hence
$$
\chi_c(X)\leq\chi_c(\Cay(\langle E'\rangle,E')).
$$

Apply the automorphism from \Cref{lem:normalize-reflection} sending
$c_s=(a_s,1)$ to $(0,1)$. By \Cref{lem:projected-semiminimal}, the
associated family in the abelian part is the finite subfamily
$$
D':=(d_\alpha)_{c_\alpha\in E',\ \alpha\neq s}.
$$
We claim that $r(D')\geq r$. For $\alpha\neq s$, set
$H_\alpha:=\langle d_\beta:\beta<\alpha,\ \beta\neq s\rangle$, so that,
as in \Cref{def:Hi(D) and r(D)},
$r_\alpha:=\operatorname{ord}_{A/H_\alpha}(d_\alpha+H_\alpha)$ is the
relative order of $d_\alpha$ in $D$, and $r=\min_\alpha r_\alpha$. For
every $\alpha$ with $c_\alpha\in E'\setminus\{c_s\}$, the predecessors
of $d_\alpha$ occurring in $D'$ — namely, those $d_\beta$ with
$\beta<\alpha$, $\beta\neq s$, and $c_\beta\in E'$ — form a subset of
$\{d_\beta:\beta<\alpha,\ \beta\neq s\}$, so the subgroup they
generate is contained in $H_\alpha$. Hence the relative order of
$d_\alpha$ in $D'$ is at least $r_\alpha\geq r$, and therefore
$r(D')\geq r$, as claimed.

The conclusion now follows from \Cref{thm:semiminimal-Dih-circular},
applied to the finite family $D'$, together with
\Cref{lem:circular-compactness}.
\end{proof}

\section{Virtually Abelian Groups}

Since $\chi(Y)= \lceil\chi_c(Y)\rceil$, \Cref{cor:arbitrary-abelian-circular} implies that for an abelian group $A$ and a semiminimal  family $C$, $\Cay(A,C)$ is $4$-colorable.  More precisely,

\begin{corollary}\label{cor:abelian}
Let $A$ be an abelian group and $
C=(c_\alpha)_{\alpha<\iota}
$
 a well-ordered semiminimal  family. If $r := r(C)$ denotes the order of $C$, then
\[ \chi(\Cay(A, C)) \leq  \left\{ \begin{array}{lll}  2 & {\rm if \ } r = \infty,  \\ 3 & {\rm if\ } r \geq 3, \\  4 & {\rm if\ } r = 2. \end{array} \right. \]
\end{corollary}


We will use this to bound the chromatic number of a semiminimal Cayley graphs of $G$ by its abelian subgroups. 
The following is a useful tool:
Let $A\leq G$ and let $S\subseteq G$ such that $S \cap A = \varnothing$.
The \emph{Schreier (coset) graph} $\Sch(G/A,S)$ is the graph whose vertices are the left
cosets of $A$ in $G$, and where two distinct cosets $gA$ and $hA$ are adjacent if
there exists $s\in S  \cup S^{-1}$ and $t\in hA$ such that $ts\in gA$.
Equivalently, $gA$ and $hA$ are adjacent whenever some edge of $\Cay(G,S)$ joins a vertex of $gA$ to a vertex of $hA$.

\begin{lemma}
\label{prop:schreier reduction}
Let $G$ be a group, let $A\leq G$, and let $C$ be a family of elements of
$G$. 
Then 
$$
\chi\bigl(\Cay(G,C)\bigr)\leq\chi\bigl(\Cay(A,C \cap A)\bigr)\chi\bigl(\Sch(G/A,C\setminus A)\bigr)
,
$$
where the product on the right-hand side is cardinal multiplication.
\end{lemma}

\begin{proof}Set $\kappa:=\chi\bigl(\Cay(A,C\cap A)\bigr)$ and
$\lambda:=\chi\bigl(\Sch(G/A,C\setminus A)\bigr)$. Choose proper colorings $\eta:G/A\longrightarrow Q$ and $\mu:A\longrightarrow K$,
where $|Q|\leq\lambda$ and $|K|\leq\kappa$.

For every left coset $gA$, choose a representative $g$ and use the
isomorphism
$$
A\longrightarrow gA,
\qquad
a\longmapsto ga,
$$
to transfer the coloring $\mu$ to the subgraph of $\Cay(G,C)$ induced by
$gA$. Denote the resulting coloring by
$
\mu_{gA}:gA\longrightarrow K.$

Define the coloring $$ \begin{array} {ccclccc} \varphi: & G & \longrightarrow & Q\times K \\ & x & \mapsto & 
\varphi(x):=\bigl(\eta(xA),\mu_{xA}(x)\bigr).\end{array}
$$
We claim that $\varphi$ is proper. Let $x,y\in G$ be adjacent in
$\Cay(G,C)$. Then $y=xc$
for some $c\in C\cup C^{-1}$.

If $c\in A$, then $xA=yA$ and the edge $xy$ belongs to the subgraph
induced by this coset. Therefore $\mu_{xA}(x)\neq\mu_{xA}(y),$ and $\varphi(x)\neq\varphi(y)$.

If $c\notin A$, then $xA$ and $yA$ are distinct adjacent vertices of $\Sch(G/A,C\setminus A).$
Therefore $\eta(xA)\neq\eta(yA),$ and again $\varphi(x)\neq\varphi(y)$.

Thus $\varphi$ is a proper coloring of $\Cay(G,C)$ with color set
$Q\times K$. Consequently,
$$
\chi\bigl(\Cay(G,C)\bigr)
\leq
|Q\times K|
\leq
\lambda\kappa
=\chi\bigl(\Cay(A,C \cap A)\bigr)\chi\bigl(\Sch(G/A,C\setminus A)\bigr).
$$
\end{proof}

From ~\Cref{cor:abelian} and  \Cref{prop:schreier reduction}, the following result is straightforward.

\begin{corollary}
\label{cor:easy 4k bound for non-Ab}
If $G$ contains an abelian subgroup $A$ of index $k$ and $C$ is a semiminimal  family of $G$, then
\[
\chi(\Cay(G,C)) \leq 4\,\chi(\Sch(G/A,C\setminus A)) \leq 4k.
\]
\end{corollary}

Recall that a class of finite groups has uniformly bounded \emph{Jordan constant} if and only if its members contain abelian subgroups of uniformly bounded index; see~\cite{MundetTurull15}. Further, by Jordan's theorem, for every $n$ there exists a constant $J(n)$ such
that every finite subgroup $G\leq \operatorname{GL}_n(\mathbb C)$ contains
an abelian normal subgroup of index at most $J(n)$; see~\cite{Col07}.
A concrete application of \Cref{cor:easy 4k bound for non-Ab} arises in semidirect products of abelian groups, as the following result shows.

\begin{proposition}
\label{thm:abelian-by-abelian-action-bound}
Let $A$ and $H$ be abelian groups, and let
$$
G=A\rtimes_{\alpha} H
$$
be a semidirect product, where $\alpha\colon H\to \operatorname{Aut}(A)$ denotes the action. For every  semiminimal family $C$ of $G$, we have
$$
\chi(\operatorname{Cay}(G,C))\le 4|\alpha(H)|.
$$
    
\end{proposition}

\begin{proof}
Let $N=\ker(\alpha)\le H$. Since $H$ is abelian, $N$ is a normal subgroup of $H$. Hence $K:=A\rtimes N$
is a normal subgroup of $G=A\rtimes_{\alpha}H$. Then, $G/K\cong H/N\cong \alpha(H),$
so $[G:K]=|\alpha(H)|.$ 

Moreover, every element of $N$ acts trivially on $A$. Therefore
$K\cong A\times N$ is abelian.

The result now follows from~\Cref{cor:easy 4k bound for non-Ab}.
\end{proof}

\section{Remarks}

It remains open whether semiminimal Cayley graphs have bounded
chromatic number. We have shown that semiminimal Cayley graphs of
abelian groups and of dihedral groups are always $4$-colorable. More
generally, by \Cref{cor:easy 4k bound for non-Ab}, semiminimal Cayley
graphs of any group with an abelian subgroup of bounded index have
bounded chromatic number. The most concrete instance of this concerns
groups with an abelian subgroup of index at most $2$, for which
\Cref{cor:easy 4k bound for non-Ab} gives $\chi(\Cay(G,C))\le 8$:
\begin{itemize}
    \item \emph{Dedekind groups.} By Dedekind's theorem, a  non-abelian Dedekind group is isomorphic to $Q_8\times A$ for
    some abelian group $A$. Although semiminimal Cayley graphs of both $Q_8$ and of abelian groups are $4$-colorable, we do not know a better bound than $8$ for Dedekind groups in  general.
    \item \emph{Generalized quaternion groups.} The generalized
    quaternion group $Q_{32}$ already has chromatic number $7$ for
    some semiminimal generating family, see~\cite[Proposition
    3.1]{GMK25}.
    \item \emph{Semidirect products $A\rtimes C_2$.} A better bound
    is known only for generalized dihedral groups, where $C_2$ acts
    by inversion; for other actions we do not know a bound better
    than $8$.
\end{itemize}

With the help of Rhys Evans, we have verified computationally that no
semiminimal Cayley graph of a group of order at most $127$ has
chromatic number greater than $7$.

By Observation \ref{obs:semimin}, every semiminimal generating family of a
subgroup $H\le G$ extends to a semiminimal generating family of $G$.
Consequently, if every semiminimal Cayley graph of $G$ has bounded
chromatic number, then so does every semiminimal Cayley graph of
$H$. Since, by Cayley's theorem, every finite group embeds as a
subgroup of some symmetric group, it would therefore suffice to
settle Babai's question for symmetric groups.

Finally, for a bound in terms of the order of $G$: no semiminimal
Cayley graph of a group $G$ contains $K_{5,17}$ as a
subgraph~\cite{B78}, and the maximum degree satisfies $\Delta\leq
2\log n$ when $|G|=n$~\cite{GMK25}. Combining these two facts
with~\cite{ABD23} yields
$$
\chi(\Cay(G,C))\leq\left(\frac{2}{\log 2}+o(1)\right)\frac{\log
n}{\log\log n}
$$
for any semiminimal  family $C$.

\subsection*{Acknowledgements}
\sloppy
GM was supported by Horizon Europe ERC Grant number: 101045750 / Project acronym: HodgeGeoComb. IGM and KK were supported through grant 
PID2022-137283NB-C22 funded by MICIU/AEI/10.13039/501100011033 by ERDF/EU. KK was also supported  through the Severo Ochoa and Mar\'ia de Maeztu Program for Centers and Units of Excellence in R\&D (CEX2020-001084-M) and ANR-21-CE48-0012.

\bibliographystyle{abbrvurl}

\bibliography{bibliography}

@Article{GMK25,
 Author = {Garc{\'{\i}}a-Marco, Ignacio and Knauer, Kolja},
 Title = {Coloring minimal {Cayley} graphs},
 FJournal = {European Journal of Combinatorics},
 Journal = {European J. Combin.},
 ISSN = {0195-6698},
 Volume = {125},
 Pages = {9},
 Note = {Id/No 104108},
 Year = {2025},
 Language = {English},
 DOI = {10.1016/j.ejc.2024.104108},
 zbMATH = {7968035}
}

@article{barajas2009chromatic,
  title={On the chromatic number of circulant graphs},
  author={Barajas, Javier and Serra, Oriol},
  journal={Discrete mathematics},
  volume={309},
  number={18},
  pages={5687--5696},
  year={2009},
  publisher={Elsevier}
}

@misc{Rhys,
  title = {Datasets of Highly Symmetric Objects},
  howpublished = {\url{https://graphsym.net/}},
  note = {Accessed: April 17, 2025},
  author = {Potočnik, Primož and Potočnik, Gregor and Evans, Rhys},
  year = {2024}
}

@Misc{B78,
 Author = {Laszlo {Babai}},
 Title = {{Chromatic number and subgraphs of {C}ayley graphs}},
 Language = {English},
 HowPublished = {{Theor. Appl. Graphs, Proc. Kalamazoo 1976, Lect. Notes Math. 642, 10-22 (1978).}},
 MSC2010 = {05C25 05C15},
 Zbl = {0382.05031}
}

@Misc{DHY24,
 Author = {Davies, James and Hatzel, Meike and Yepremyan, Liana},
 Title = {Counterexample to {Babai}'s lonely colour conjecture},
 Year = {2024},
 HowPublished = {Preprint, {arXiv}:2410.05199 [math.{CO}] (2024)},
 URL = {https://arxiv.org/abs/2410.05199},
 arXiv = {arXiv:2410.05199}
}

@Article{GMK22,
 Author = {Garc{\'{\i}}a-Marco, Ignacio and Knauer, Kolja},
 Title = {On sensitivity in bipartite {Cayley} graphs},
 FJournal = {Journal of Combinatorial Theory. Series B},
 Journal = {J. Comb. Theory, Ser. B},
 ISSN = {0095-8956},
 Volume = {154},
 Pages = {211--238},
 Year = {2022},
 Language = {English},
 DOI = {10.1016/j.jctb.2022.01.002},
 zbMATH = {7483272},
 Zbl = {1483.05078}
}

@article{KS25,
 author = {Knauer, Kolja and Soto G{\'o}mez, {\'A}lvaro},
 title = {What is and is not inside a {Cayley} graph?},
 fjournal = {DML. Discrete Mathematics Letters},
 journal = {DML, Discrete Math. Lett.},
 issn = {2664-2557},
 volume = {16},
 pages = {67--72},
 year = {2025},
 language = {English},
 doi = {10.47443/dml.2025.111},
 zbMATH = {8118333}
}

@Article{God81,
 Author = {Godsil, C. D.},
 Title = {Connectivity of minimal {Cayley} graphs},
 FJournal = {Archiv der Mathematik},
 Journal = {Arch. Math.},
 ISSN = {0003-889X},
 Volume = {37},
 Pages = {473--476},
 Year = {1981},
 Language = {English},
 DOI = {10.1007/BF01234384},
 zbMATH = {3747168},
 Zbl = {0476.05049}
}

@Article{Spe83,
 Author = {Spencer, Joel},
 Title = {What's not inside a {Cayley} graph},
 FJournal = {Combinatorica},
 Journal = {Combinatorica},
 ISSN = {0209-9683},
 Volume = {3},
 Pages = {239--241},
 Year = {1983},
 Language = {English},
 DOI = {10.1007/BF02579297},
 zbMATH = {3825843},
 Zbl = {0522.05032}
}

@InCollection{B95,
 Author = {L\'aszl\'o {Babai}},
 Title = {{Automorphism groups, isomorphism, reconstruction}},
 BookTitle = {{Handbook of combinatorics. Vol. 1-2}},
 ISBN = {0-444-88002-X/set; 0-444-82346-8/vol1; 0-444-82351-4/vol2; 0-262-07169-X/set; 0-262-07170-3/vol1; 0-262-07171-1/vol2},
 Pages = {1447--1540},
 Year = {1995},
 Publisher = {Amsterdam: Elsevier (North-Holland); Cambridge, MA: MIT Press},
 Language = {English},
 MSC2010 = {05C25 05C60 05C10},
 Zbl = {0846.05042}
}

@article{Vin88,
 author = {Vince, A.},
 title = {Star chromatic number},
 fjournal = {Journal of Graph Theory},
 journal = {J. Graph Theory},
 issn = {0364-9024},
 volume = {12},
 number = {4},
 pages = {551--559},
 year = {1988},
 language = {English},
 doi = {10.1002/jgt.3190120411},
 zbMATH = {4075092},
 Zbl = {0658.05028}
}

@article{MundetTurull15,
  author  = {Mundet i Riera, Ignasi and Turull, Alexandre},
  title   = {Boosting an analogue of Jordan's theorem for finite groups},
  journal = {Advances in Mathematics},
  volume  = {272},
  year    = {2015},
  pages   = {820--836},
  doi     = {10.1016/j.aim.2014.12.021}
}

@article{Zhu01,
 author = {Zhu, Xuding},
 title = {Circular chromatic number: {A} survey},
 fjournal = {Discrete Mathematics},
 journal = {Discrete Math.},
 issn = {0012-365X},
 volume = {229},
 number = {1-3},
 pages = {371--410},
 year = {2001},
 language = {English},
 doi = {10.1016/S0012-365X(00)00217-X},
 zbMATH = {1591126},
 Zbl = {0973.05030}
}

@article{ABD23,
 author = {Anderson, James and Bernshteyn, Anton and Dhawan, Abhishek},
 title = {Colouring graphs with forbidden bipartite subgraphs},
 fjournal = {Combinatorics, Probability and Computing},
 journal = {Comb. Probab. Comput.},
 issn = {0963-5483},
 volume = {32},
 number = {1},
 pages = {45--67},
 year = {2023},
 language = {English},
 doi = {10.1017/S0963548322000104},
 zbMATH = {7671448},
 Zbl = {1511.05068}
}

@book{MunkresTopology,
  author    = {Munkres, James R.},
  title     = {Topology},
  edition   = {2},
  publisher = {Prentice Hall},
  year      = {2000}
}

@article{Col07,
  author  = {Collins, Michael J.},
  title   = {On Jordan's theorem for complex linear groups},
  journal = {Journal of Group Theory},
  volume  = {10},
  number  = {4},
  year    = {2007},
  pages   = {411--423},
  doi     = {10.1515/JGT.2007.032}
}

@article{deBruijnErdos1951,
  author  = {de Bruijn, N. G. and Erd{\H{o}}s, P.},
  title   = {A Colour Problem for Infinite Graphs and a Problem in the Theory of Relations},
  journal = {Nederl. Akad. Wetensch. Proc. Ser. A},
  volume  = {54},
  pages   = {371--373},
  year    = {1951},
  note    = {Also published in Indagationes Mathematicae, 13:371--373}
}

\end{document}